\documentclass[10pt]{article}
\usepackage{amstext}
\usepackage{amssymb}
\usepackage{latexsym}
\usepackage{amsmath}
\usepackage{amscd}
\usepackage{amsthm}
\usepackage{graphicx}

\newtheorem{theorem}{Theorem}[section]
\newtheorem{proposition}[theorem]{Proposition}
\newtheorem{lemma}[theorem]{Lemma}
\newtheorem{corollary}[theorem]{Corollary}

\newtheorem{remark}[theorem]{Remark}

\newtheorem{notation}[theorem]{Notation}

\newcommand{\F}{\mathbb{F}}

\newcommand{\Z}{\mathbb{Z}}
\newcommand{\Q}{\mathbb{Q}}

\newcommand{\bir}{{\rm Bir}}

\title{Polynomial parametrizations of length $4$ B\"uchi sequences}

\date{}

\author{Xavier Vidaux\\Universidad de Concepci\'on, Chile}

\begin{document}

\maketitle

\begin{abstract}
B\"uchi's problem asks whether there exists a positive integer $M$ such that any sequence $(x_n)$ of at least $M$ integers, whose second difference of squares is the constant sequence $(2)$, satisifies $x_n^2=(x+n)^2$ for some $x\in\Z$. A positive answer to B\"uchi's problem would imply that there is no algorithm to decide whether or not an arbitrary system of quadratic diagonal forms over $\Z$ can represent an arbitrary given vector of integers. We give explicitly an infinite family of polynomial parametrizations of non-trivial length $4$ B\"uchi sequences of integers. In turn, these parametrizations give an explicit infinite family of curves (which we suspect to be hyperelliptic) with the following property: any integral point on one of these curves would give a length $5$ non-trivial B\"uchi sequence of integers (it is not known whether any such sequence exists). 
\end{abstract}

\tableofcontents

\section{Introduction}

Since the projective surface with affine equations
\begin{equation}\label{buc4}
x_4^2-2x_3^2+x_2^2=x_3^2-2x_2^2+x_1^2=2 
\end{equation}
is a Segre surface (a Del Pezzo surface of degree $4$), its $\Q$-rational points can be parametrized. Let $X_4$ be the (affine) variety defined by Equations \eqref{buc4} (\emph{B\"uchi equations}). Having B\"uchi's problem in mind, we would like to characterize the set of integer points on $X_4$ (actually a cofinite set of the set of integer points would be enough). There exists extensive literature about rational surfaces, but there seem to be few results about polynomial parametrizations over $\Z$. 

Let $A$ be a commutative ring with unit. A sequence $(x_1,\dots,x_\ell)$ of elements of $A$ is called a \emph{B\"uchi sequence over $A$} if its second difference is the constant sequence $(2)$\,: for each $i\in\{1,\dots,\ell-2\}$ it satisfies
$$
x_i^2-2x_{i+1}^2+x_{i+2}^2=2.
$$
If $A$ has characteristic $0$, we will call \emph{trivial B\"uchi sequence} any sequence satisfying\,: there exists $x\in A$ such that $x_i^2=(x+i)^2$ for all $i=1,\dots,\ell$. B\"uchi's problem over $A$ asks whether there exists an integer $M$ such that no non-trivial B\"uchi sequence of length at least $M$ exists. If such an $M$ exists, we call $M(A)$ the smallest one, and $M_f(A)$ the least $M$ such that there are only finitely many non-trivial B\"uchi sequences of length $M$. Hence, if one proves that $M_f(A)$ exists, then one obtains automatically a positive answer to B\"uchi's problem \emph{for some} $M\ge M_f(A)$.

If $A$ is a ring of functions, then it is required that at least one $x_i$ is transcendental over the prime subring of $A$ (also in the case where $A$ has positive characteristic, the concept of trivial sequence has to be adapted - see for example \cite{PastenPheidasVidaux}). 

B\"uchi got interested in this problem when he realized that from a positive answer to it he would be able to prove that there is no algorithm to decide whether or not an arbitrary system of quadratic diagonal forms can represent an arbitrary given vector of integers (which, if true, would be one of the strongest forms of the negative answer to Hilbert's tenth problem). 

B\"uchi's problem remains open for the integers, but P. Vojta \cite{Vojta} showed that $M_f(\Q)$ would be $8$ (actually the proof goes through for any number field) if Bombieri conjecture would be true for surfaces. It is striking that even though we cannot prove that B\"uchi's problem has a positive answer, no non-trivial B\"uchi sequence of length even just $5$ over $\Z$ is known to exist. Actually B\"uchi had conjectured that $M(\Z)=5$. 

B\"uchi's problem is known to have a positive answer over many usual rings of functions\,: fields of complex and $p$-adic complex meromorphic functions (\cite{Vojta}, \cite{Pasten3}), function fields in characteristic zero or large enough (\cite{Vojta},\cite{PheidasVidaux2},\cite{PheidasVidaux3},\cite{ShlapentokhVidaux}). Most of these results have been generalized to stronger forms in \cite{Pasten} and \cite{Pasten3}. A higher power version of B\"uchi's problem was introduced in \cite{PheidasVidaux1} (looking at $k$-th difference of $k$-th powers), but very few results have been obtained so far\,: it has a positive answer for any power over finite fields $\F_p$ - see \cite{Pasten2} - and for cubes over rings of polynomials - see \cite{PheidasVidaux4}. We refer to the surveys \cite{PastenPheidasVidaux} for results in these directions. 

In \cite{Allison}, \cite{Bremner} and \cite{BrowkinBrzezinski} one can find results about analogues of B\"uchi's problem where the constant sequence $(2)$ is replaced by another constant sequence. 

B\"uchi sequences of length $3$ are not difficult to characterize over $\Q$, and with some divisibility conditions one obtains a complete characterization of sequences over $\Z$ - they are infinitely many - see \cite[Theorem 2.1]{Hensley2} or \cite[Section 7]{PastenPheidasVidaux}. We also know a characterization over $\Z$ that does not require any divisibility condition (i.e.\: without any reference to $\Q$) - see \cite{SaezVidaux}. 

Already Hensley \cite{Hensley2} knew that there exist infinitely many length $4$ B\"uchi sequences of integers. Obtaining a ``good'' characterization for (a cofinite subset of the set of all) length $4$ sequences of integers could be a key step for solving B\"uchi's problem\,: proving that no sequences of length $4$ (but finitely many) can be extended to length $5$ could then be quite easier, and would prove that $M_f(\Z)=5$. 

This work presents an effort to characterize all but finitely many B\"uchi sequences of length $4$ over the integers. The idea comes from an unpublished paper by D. Hensley \cite{Hensley2} from the early eighties, where a polynomial parametrization of degree $3$ for length $4$ integer sequences is described, and from a paper by R. G. E. Pinch \cite{Pinch} from 1993 where he lists many length $4$ non-trivial B\"uchi sequences and shows that none of them can be extended to length $5$ sequences. 

In Section \ref{bir} we give explicitly a birational map $\zeta$ on $X_4$, of infinite order, with the following property (proved in Section \ref{xi})\,: 

\begin{theorem}
For any $t\in\Z$ and any non-negative integer $n$, the $n$-th iterate
$$
\zeta^{(n)}(t,t+1,t+2,t+3)
$$ 
has the form 
$$
\xi(n,t)=(\xi_{1}(n,t),\xi_{2}(n,t),\xi_{3}(n,t),\xi_{4}(n,t))
$$
where each $\xi_i(n,t)$ is a polynomial in $t$ of degree $2n+1$. 
\end{theorem}

The polynomials $\xi_i(n,t)\in\Z[t]$ are given explicitly in Section \ref{xi}\,: they satisfy the following linear second order recurrence relation\,: for all $n\geq0$ 
$$
\xi(n+2,t)=f(t)\xi(n+1,t)-\xi(n,t),
$$
where $f(t)=2t^{2}+10t+10$, and with initial values
$$
\begin{aligned}
\xi(0,t)&=(t+1,t+2,t+3,t+4)\\
\xi(1,t)&=(2t^3+12t^2+19t+6,2t^3+14t^2+31t+23,2t^3+16t^2+41t+32,2t^3+18t^2+49t+39).
\end{aligned}
$$

In other words, there are infinitely many non-trivial parametrizations of B\"uchi sequences of length $4$ over the integers (one can prove that they are \emph{essentially} distinct, as they give rise to distinct sequences). Iterating $\zeta$ on known rational parametrizations over $\Q$ gives rise to infinitely many rational parametrizations (points of $X_4(\Q(t))$). 

The parametrization $\xi(1,t)$ is actually the one that D. Hensley \cite{Hensley2} had found in the early eighties, and no other polynomial parametrization seemed to be known to this day. 

In Section \ref{other}, we present two more polynomial parametrizations over $\Z$, of degree $4$, and one polynomial parametrization over $\Q$, also of degree $4$. Then we give a non-exhaustive  list of rational parametrizations which are not polynomial parametrizations. Computationally, it seems that there are no other polynomial parametrizations than the ones we already found, but we cannot prove it. 

Also we cannot prove the following, which seems to be true computationally\,: none of these parametrizations can represent an integer solution that extends to a length $5$ B\"uchi sequence. Indeed, consider for example the sequence $\xi(1,t)=(x_1(t),x_2(t),x_3(t),x_4(t))$ given above. Asking for this sequence to extend, for some fixed integer $t$, to a length $5$ sequence is asking whether either $2x_4^2-x_3^2+2$ (extension to the right) or $2x_1^2-x_2^2+2$ (extension to the left) is a square. Namely: do one of the following curves
$$
y^2=4t^6+80t^5+620t^4+2400t^3+4905t^2+5020t+2020
$$
or 
$$
y^2=4t^6+40t^5+120t^4-595t^2-970t-455
$$
have an integer point? 

In Section \ref{list}, we list all integer solutions that we found and that we were not able to \emph{parametrize} (i.e.: they seem not to belong to the image of a polynomial parametrization). With first term at most $1052749$, they are $121$ (counting only the strictly increasing sequences of positive integers) and we do not know whether or not we are missing finitely many. From the figure at the end of the section, it \emph{seems} clear that the quantity of points that we are ``missing'' is decreasing exponentially with respect to the size of the points. None of these (non-parametrized) points can extend to a length $5$ solution, as is easily verified with a computer software. 

The symbol $\dagger$ in the text will mean that we are using a computer software for the formal computation. We have used exclusively the open source software Xcas 0.8.6 and 0.9.0 for all our computations: \\
Giac/Xcas, Bernard Parisse et RenŽe De Graeve, version 0.8.6 (2010)\\
http://www-fourier.ujf-grenoble.fr/$\sim$parisse/giac\_fr.html\\




\section{Some birational maps on $X_4$}\label{bir}

\begin{notation}
\begin{enumerate}
\item Denote by $\bir(X_4)$ the group of birational maps on $X_4$.
\item Let $\tau$ and $\mu_i$, $i=1,2,3,4$, denote the following automorphisms of $X_4$\,:
$$
\mu_1(a,b,c,d)=(-a,b,c,d)\qquad\mu_2(a,b,c,d)=(a,-b,c,d)
$$
$$
\mu_3(a,b,c,d)=(a,b,-c,d)\qquad\mu_4(a,b,c,d)=(a,b,c,-d)
$$
and
$$
\tau(a,b,c,d)=(d,c,b,a).
$$
Observe that each $\mu_i$ is an odd function.
\item We will call \emph{trivial involution on $X_4$} any map from the subgroup $\Gamma_1$ of $\bir(X_4)$ generated by the set $\{\mu_1,\mu_2,\mu_3,\mu_4,\tau\}$. 
\item Write $\Gamma_0$ for the group generated by $\{\mu_1,\mu_2,\mu_3,\mu_4\}$. 
\item Write $\mu_{ij}=\mu_i\mu_j$ and $\mu_{ijk}=\mu_i\mu_j\mu_k$ for any $i,j,k\in\{1,2,3,4\}$.
\end{enumerate}
\end{notation}

\begin{remark}
\begin{enumerate}
\item For all $i\ne j$ we have $\mu_i\mu_j=\mu_j\mu_i$, hence $\Gamma_0$ is isomorphic to $(\Z_2)^4$. 
\item We have $\tau\mu_1=\mu_4\tau$ and $\tau\mu_2=\mu_3\tau$.
\item We have $\tau\mu_{14}=\mu_{14}\tau$ and $\tau\mu_{23}=\mu_{23}\tau$.
\item For each $i$, $\tau\mu_i$ has order $4$. 
\end{enumerate}
\end{remark}

\begin{lemma}
For all $i$, we have $\tau\mu_i\tau=\mu_{\sigma(i)}$, where $\sigma$ stands for the permutation $(1\,\,4)(2\,\,3)\in S_4$. Hence the group $H$ is normal in $\Gamma_1$ and 
the group $\Gamma_1$ is a semi-direct product $\Gamma_0\rtimes <\tau>$.
\end{lemma}
\begin{proof}
This is clear from the above remarks. 
\end{proof}

Next we define a rational map $\varphi$ on $X_4$ that will turn out to be an involution. 

\begin{notation}\label{SMALL}
\begin{enumerate}
\item We will consider the map $\varphi\colon F^4\rightarrow F^4$ defined by
$$
(\varphi_1,\varphi_2,\varphi_3,\varphi_4)=\left(\frac{p_1}{q},\frac{p_2}{q},\frac{p_3}{q},\frac{p_4}{q}\right)
$$
where
$$
q(a,b,c,d)=(b-c)^2(a-2b+c)
$$
\begin{multline}
p_1(a,b,c,d)=-2 a b^3+a b^2 c+2 a b^2 d+4 a b c^2-5 a b c d+a b-2 a c^3+2 a c^2 d+a c-a d+3 b^4
\\\notag 
-2 b^3 c-3 b^3 d-6 b^2 c^2+8 b^2 c d+b^2+4 b c^3-4 b c^2 d-5 b c+b d+c^2+c d-2
\end{multline}
\begin{multline}
p_2(a,b,c,d)=-2 a b^2 c+5 a b c^2-2 a b c d+2 a b-2 a c^3+a c^2 d\\\notag 
-a d+3 b^3 c-8 b^2 c^2+3 b^2 c d-2 b^2+4 b c^3-2 b c^2 d-b c+2 b d-2
\end{multline}
\begin{multline}
p_3(a,b,c,d)=-2 a b^3+5 a b^2 c-2 a b^2 d-2 a b c^2+a b c d\\\notag 
+3 a b-a c-a d+3 b^4-8 b^3 c+3 b^3 d+4 b^2 c^2-2 b^2 c d-3 b^2-b c+3 b d+c^2-c d-2
\end{multline}
\begin{multline}
p_4(a,b,c,d)=-3 a b^3+8 a b^2 c-3 a b^2 d-4 a b c^2+2 a b c d+4 a b-2 a c-a d+4 b^4-10 b^3 c+4 b^3 d\\\notag
+2 b^2 c^2-2 b^2 c d-4 b^2+5 b c^3-2 b c^2 d-b c+4 b d-2 c^4+c^3 d+2 c^2-2 c d-2.
\end{multline}
Observe that $\varphi$ is an odd function (as $q$ is an odd function and $p_i$ are even functions).
\item Write $\zeta=\varphi\tau\mu_{14}$.
\end{enumerate}
\end{notation}

\begin{lemma}
The map $\varphi$ is a rational map on $X_4$. 
\end{lemma}
\begin{proof}
One needs to replace formally ($\dagger$) $a^2$ by $2b^2-c^2+2$ and $d^2$ by $2c^2-b^2+2$ in the expressions $(\varphi_1^2-2\varphi_2^2+\varphi_3^2)(a,b,c,d)$ and $(\varphi_2^2-2\varphi_3^2+\varphi_4^2)(a,b,c,d)$.
\end{proof}


\begin{lemma}\label{LemInv}
The map $\varphi$ is an involution. 
\end{lemma}
\begin{proof}
For each $i$, after substituting formally ($\dagger$) $x_4^2$ by $2x_3^2-x_2^2+2$ and $x_3^2$ by $2x_2^2-x_1^2+2$ in $\varphi_i(\varphi(x_1,x_2,x_3,x_4))$ and doing the obvious simplifications ($\dagger$), one obtains $x_i$. Note that it is not hard to prove this lemma without the help of a computer, by using the fact (verifiable by hand) that 
$$
(\varphi_1-2\varphi_2+\varphi_3)(a,b,c,d)=a-2b+c
$$
and 
$$
(\varphi_2-2\varphi_3+\varphi_4)(a,b,c,d)=b-2c+d.
$$

\end{proof}

Observe that since $\varphi$ is birational, also $\zeta=\varphi\tau\mu_{14}$ is birational. 

\begin{notation}
Write $\Gamma$ for the subgroup of $\bir(X_4)$ generated by $\Gamma_1$ and $\varphi$.
\end{notation}

Unfortunately, we do not know the presentation of $\Gamma$, but the next two lemmas give us some useful information about it. 

\begin{lemma}
We have $\tau\varphi=\varphi\tau$ and $\tau\zeta=\zeta\tau$. 
\end{lemma}
\begin{proof}
Verifying that $\tau\varphi-\varphi\tau=0$ needs replacing $a^2$ by $2b^2-c^2+2$ everywhere it occurs in the expression ($\dagger$). Recalling the definition of $\zeta=\varphi\tau\mu_{14}$, we have 
$$
\tau\zeta\tau=\tau(\varphi\tau\mu_{14})\tau=\varphi\mu_{14}\tau=\varphi\tau\mu_{14}=\zeta.
$$
\end{proof}

\begin{lemma}
The map $\zeta$ has infinite order. 
\end{lemma}
\begin{proof}
The sequence $(u_n)$ defined for $n\geq0$ by
\begin{equation}\label{EqSol1}
u_n=\frac{(5+\sqrt{24})^n-(5-\sqrt{24})^n}{2\sqrt{24}}
\end{equation}
is strictly increasing (since $(5+\sqrt{24})^n$ is strictly increasing and $0<5-\sqrt{24}<1$). Also it
satisfies ($\dagger$) the recurrence relation 
$$
u_{n+2}=10u_{n+1}-u_{n},\qquad u_0=0,\,\,\,u_1=1
$$ 
(hence it is a sequence of integers). Therefore, we have 
\begin{equation}\label{EqRec1}
0<u_{n+3}-u_{n+2}=10(u_{n+2}-u_{n+1})-(u_{n+1}-u_{n})
\end{equation}

Let us show that for any $n\ge1$, we have  
\begin{equation}\label{EqRecZeta1}
\zeta^{(n)}(1,2,3,4)=u_{n}(6,23,32,39)-u_{n-1}(1,2,3,4). 
\end{equation}
Since ($\dagger$)
$$
\zeta(1,2,3,4)=(6,23,32,39)\quad\textrm{and}\quad\zeta^{(2)}(1,2,3,4)=(59,228,317,386),
$$
we have  
$$
\zeta^{(2)}(1,2,3,4)=10(6,23,32,39)-(1,2,3,4)
$$
and Equation \eqref{EqRecZeta1} is true for $n=2$. Suppose that it is true up to $n$. We have 
$$
\begin{aligned}
\zeta^{(n+1)}(1,2,3,4)&=\zeta\circ\zeta^{(n)}(1,2,3,4)\\
&=\zeta(u_{n}(6,23,32,39)-u_{n-1}(1,2,3,4))\\
&=u_{n+1}(6,23,32,39)-u_{n}(1,2,3,4),
\end{aligned}
$$
where the last equality can be verified with the use of a computer software. This finishes the induction. 

Since the second component of $(\zeta^{(n+1)}-\zeta^{(n)})(1,2,3,4)$ is
$$
\begin{aligned}
(23u_{n+2}-2u_{n+1})-(23u_{n+1}-2u_{n})&=23(u_{n+2}-u_{n+1})-2(u_{n+1}-u_{n})\\
&=3(u_{n+2}-u_{n+1})+2[10(u_{n+2}-u_{n+1})-(u_{n+1}-u_{n})]
\end{aligned}
$$
which is strictly positive by Equation \eqref{EqRec1}, this shows that $\zeta$ has infinite order. 
\end{proof}

Note that from the above proof we also obtain\,:
$$
\begin{aligned}
\zeta^{(n)}(1,2,3,4)&=u_{n}(6,23,32,39)-u_{n-1}(1,2,3,4)\\
&=(10u_{n-1}-u_{n-2})(6,23,32,39)-(10u_{n-2}-u_{n-1})(1,2,3,4)\\
&=10(u_{n-1}(6,23,32,39)-u_{n-2}(1,2,3,4))-(u_{n-2}(6,23,32,39)-u_{n-1}(1,2,3,4))\\
&=10\zeta^{(n-1)}(1,2,3,4)-\zeta^{(n-2)}(1,2,3,4)
\end{aligned}
$$
for all $n\geq2$.

\section{B\"uchi sequences of length $4$ over $\Z[t]$}\label{xi}

\begin{notation}
For all $n\geq0$, write $\xi(n,t)=(\xi_{1}(n,t),\xi_{2}(n,t),\xi_{3}(n,t),\xi_{4}(n,t))$ where 
$$
\xi_{1}(n,t)=\frac{(t^3+6t^2+9t+1+\alpha (t+1))\beta^n-
(t^3+6t^2+9t+1-\alpha (t+1))\bar\beta^n}{2\alpha}
$$
$$
\xi_{2}(n,t)=\frac{(t^3+7t^2+16t+13+\alpha (t+2))\beta^n-
(t^3+7t^2+16t+13-\alpha (t+2))\bar\beta^n}{2\alpha}
$$
$$
\xi_{3}(n,t)=\frac{(t^3+8t^2+21t+17+\alpha (t+3))\beta^n-
(t^3+8t^2+21t+17-\alpha (t+3))\bar\beta^n}{2\alpha}
$$
$$
\xi_{4}(n,t)=\frac{(t^3+9t^2+24t+19+\alpha (t+4))\beta^n-
(t^3+9t^2+24t+19-\alpha (t+4))\bar\beta^n}{2\alpha}
$$
$$
\alpha=\sqrt{(t+1)(t+2)(t+3)(t+4)}
$$
and 
$$
\beta=t^2+5t+5+\alpha,\qquad \bar\beta=t^2+5t+5-\alpha.
$$
\end{notation}

\begin{theorem}\label{bigteo}
We have\,: 
\begin{enumerate}
\item $\xi(0,t)=(t+1,t+2,t+3,t+4)$.\label{bigteon0}
\item $\xi(1,t)=(2t^3+12t^2+19t+6,2t^3+14t^2+31t+23,2t^3+16t^2+41t+32,2t^3+18t^2+49t+39)$.\label{bigteon1}
\item For all $n\geq0$ 
$$
\xi(n+2,t)=f(t)\xi(n+1,t)-\xi(n,t),
$$
where $f(t)=2t^{2}+10t+10$.\label{bigteorec}
\item For each $n\geq0$, $\xi(n,t)$ is a $4$-tuple of polynomials of degree $2n+1$ in the variable $t$. 
\item For all $n\geq0$ and $t\in\Z$, the sequence $\xi(n,t)$ is a B\"uchi sequence. 
\item For all $n\geq0$ we have 
$$
\zeta^{(n)}(t+1,t+2,t+3,t+4)=\xi(n,t).
$$
\end{enumerate}
\end{theorem}
\begin{proof}
Items 1, 2, 3 and 5 are easily verified $(\dagger)$. Item 4 then comes immediately from Items 1, 2 and 3 by induction on $n$. We prove Item 6 by induction on $n$. For $n=0$ it is given by Item 1. Suppose it is true up to $n$. One verifies $(\dagger)$ that 
$$
\xi(n+1,t)=\zeta(\xi(n,t))
$$
hence 
$$
\xi(n+1,t)=\zeta(\zeta^{(n)}(t+1,t+2,t+3))=\zeta^{(n+1)}(t+1,t+2,t+3,t+4),
$$
which finishes the induction. 

\end{proof}

In the next proposition, we solve the induction in order to find the coefficients of the polynomials $\xi_i(n,t)$.

\begin{proposition}
If $(u_n)$ is a sequence of integers satisfying $u_{n+2}=\alpha u_{n+1}-u_n$ for each $n\geq0$, then we have for each $n\ge1$
$$
u_{2n}=\sum_{k=0}^{n-1}(-1)^{n+k+1}\alpha^{2k}\left(\binom{n+k}{n-k-1}\alpha u_1-\binom{n+k-1}{n-k-1}u_0\right)
$$
and
$$
u_{2n-1}=\alpha^{2n-2}u_2+\sum_{k=0}^{n-2}(-1)^{n+k+1}\alpha^{2k}\left(\binom{n+k-1}{n-k-1}u_1+\binom{n+k-1}{n-k-2}\alpha u_0\right).
$$
\end{proposition}
\begin{proof}
The proof is easy and left to the reader.  
\end{proof}

For example, applying the above proposition to $u_n=\xi_1(n,t)$ and $\alpha=f(t)=2t^2+10t+10$, we obtain
$$
\begin{aligned}
\xi_1(2n,t)=&\sum_{k=0}^{n-1}(-1)^{n+k+1}2^{2k}(t^2+5t+5)^{2k}\\
&\left(\binom{n+k}{n-k-1}(4t^5+44t^4+178t^3+322t^2+250t+60)-\binom{n+k-1}{n-k-1}(t+1)\right).
\end{aligned}
$$

We finish this section with a list of the first three non-trivial parametrizations:
$$
\begin{aligned}
\xi_1(1,t)&=2t^3+12t^2+19t+6\\
\xi_2(1,t)&=2t^3+14t^2+31t+23\\
\xi_3(1,t)&=2t^3+16t^2+41t+32\\
\xi_4(1,t)&=2t^3+18t^2+49t+39
\end{aligned}
$$

$$
\begin{aligned}
\xi_1(2,t)&=4t^5+44t^4+178t^3+322t^2+249t+59\\
\xi_2(2,t)&=4t^5+48t^4+222t^3+496t^2+539t+228\\
\xi_3(2,t)&=4t^5+52t^4+262t^3+634t^2+729t+317\\
\xi_4(2,t)&=4t^5+56t^4+298t^3+748t^2+879t+386
\end{aligned}
$$

$$
\begin{aligned}
\xi_1(3,t)&=8t^7+128t^6+836t^5+2864t^4+5496t^3+5816t^2+3061t+584\\
\xi_2(3,t)&=8t^7+136t^6+964t^5+3692t^4+8256t^3+10792t^2+7639t+2257\\
\xi_3(3,t)&=8t^7+144t^6+1084t^5+4408t^4+10416t^3+14248t^2+10419t+3138\\
\xi_4(3,t)&=8t^7+152t^6+1196t^5+5036t^4+12216t^3+17024t^2+12601t+3821
\end{aligned}
$$


\section{Other parametrizations}\label{other}

Note that by replacing $t$ by $t^2$ in a polynomial parametrization of degree $n$, we obtain a polynomial parametrization of degree $2n$. Since there exist non-trivial polynomial parametrizations of any odd degree, there are non-trivial polynomial parametrization of any degree $>2$. 

The following 
$$
\begin{aligned}
P_1(t)=&\frac{t^4+17t^3+104t^2+262t+204}{4}\\
P_2(t)=&\frac{t^4+19t^3+138t^2+458t+592}{4}\\
P_3(t)=&\frac{t^4+21t^3+168t^2+602t+812}{4}\\
P_4(t)=&\frac{t^4+23t^3+194t^2+718t+984}{4}
\end{aligned}
$$
gives a polynomial parametrization $P=(P_1,P_2,P_3,P_4)$ over $\Q$ that takes an integer value for each integer $t$ not congruent to $3$ modulo $4$. Hence one obtains two \emph{new} polynomial parametrizations over $\Z$ as follows\,:
$$
\begin{aligned}
P_1(2t)=&4t^4+34t^3+104t^2+131t+51\\
P_2(2t)=&4t^4+38t^3+138t^2+229t+148\\
P_3(2t)=&4t^4+42t^3+168t^2+301t+203\\
P_4(2t)=&4t^4+46t^3+194t^2+359t+246
\end{aligned}
$$
and 
$$
\begin{aligned}
P_1(4t+1)=&64t^4+336t^3+644t^2+525t+147\\
P_2(4t+1)=&64t^4+368t^3+804t^2+795t+302\\
P_3(4t+1)=&64t^4+400t^3+948t^2+1005t+401\\
P_4(4t+1)=&64t^4+432t^3+1076t^2+1179t+480.
\end{aligned}
$$
They are new in the sense that they generate B\"uchi sequences of integers that were not in the image of any of the $\xi(n,t)$.

The following is another polynomial parametrization over $\Q$, but it does not reach any integer solution\,:
$$
\begin{aligned}
&\frac{1}{3}\left(4t^4+18t^3+14t^2-15t-8\right)\\
&\frac{1}{3}\left(4t^4+22t^3+36t^2+19t+5\right)\\
&\frac{1}{3}\left(4t^4+26t^3+54t^2+35t+2\right)\\
&\frac{1}{3}\left(4t^4+30t^3+68t^2+45t+1\right).
\end{aligned}
$$

Next we list some parametrizations $R_1,\dots,R_{15}$ in $\Q(t)\smallsetminus\Z[t]$ (some of them have in their image some B\"uchi sequences of integers, but always finitely many). Note that the four components of each parametrization always come with the same denominator (as reduced fractions). Note also that each of these parametrizations generate infinitely many new parametrizations by iterating the map $\zeta$. 

$$
(4t^2+4)R_1(t)=
\begin{cases}
&t^9+4t^8-t^7-10t^6+9t^5-23t^3+10t^2-10t-4\\
&t^9+4t^8+t^7-4t^6+13t^5+8t^4-9t^3+16t^2+2t\\
&t^9+4t^8+3t^7+2t^6+13t^5+8t^4-3t^3+6t^2+10t-4\\
&t^9+4t^8+5t^7+8t^6+9t^5+19t^3+4t^2-10t+8
\end{cases}
$$

$$
(4t^2+4)R_2(t)=
\begin{cases}
&t^9-4t^8-t^7+10t^6+9t^5-23t^3-10t^2-10t+4\\
&t^9-4t^8+t^7+4t^6+13t^5-8t^4-9t^3-16t^2+2t\\
&t^9-4t^8+3t^7-2t^6+13t^5-8t^4-3t^3-6t^2+10t+4\\
&t^9-4t^8+5t^7-8t^6+9t^5+19t^3-4t^2-10t-8
\end{cases}
$$

$$
8tR_3(t)=
\begin{cases}
&t^6-10t^5+38t^4-84t^3+120t^2-96t+48\\
&t^6-8t^5+30t^4-68t^3+96t^2-88t+48\\
&t^6-6t^5+18t^4-44t^3+72t^2-80t+48\\
&t^6-4t^5+2t^4+12t^3-48t^2+72t-48
\end{cases}
$$

$$
4tR_4(t)=
\begin{cases}
&t^6-17t^5+113t^4-369t^3+600t^2-420t+96\\
&t^6-15t^5+93t^4-295t^3+480t^2-352t+96\\
&t^6-13t^5+69t^4-193t^3+312t^2-284t+96\\
&t^6-11t^5+41t^4-39t^3-96t^2+216t-96
\end{cases}
$$

$$
(4t^2-28)R_5(t)=
\begin{cases}
&t^7-6t^6-3t^5+60t^4-33t^3-162t^2+107t+60\\
&t^7-4t^6-5t^5+36t^4-17t^3-88t^2+93t+32\\
&t^7-2t^6-11t^5+12t^4+55t^3-14t^2-117t+4\\
&t^7-21t^5+12t^4+111t^3-60t^2-163t+24
\end{cases}
$$

$$
(4t^2-4t)R_6(t)=
\begin{cases}
&t^7-16t^6+98t^5-300t^4+506t^3-480t^2+240t-48\\
&t^7-14t^6+80t^5-240t^4+406t^3-396t^2+212t-48\\
&t^7-12t^6+58t^5-156t^4+270t^3-296t^2+184t-48\\
&t^7-10t^6+32t^5-24t^4-70t^3+180t^2-156t+48
\end{cases}
$$

$$
(8t^2-16t)R_7(t)=
\begin{cases}
&t^7-14t^6+86t^5-300t^4+644t^3-840t^2+624t-192\\
&t^7-12t^6+70t^5-240t^4+516t^3-688t^2+544t-192\\
&t^7-10t^6+50t^5-156t^4+340t^3-504t^2+464t-192\\
&t^7-8t^6+26t^5-24t^4-76t^3+288t^2-384t+192
\end{cases}
$$

$$
(8t^3+48t^2+32t)R_8(t)=
\begin{cases}
&t^8+12t^7+50t^6+96t^5+144t^4+240t^3+304t^2+288t+128\\
&t^8+14t^7+78t^6+240t^5+488t^4+664t^3+608t^2+384t+128\\
&t^8+16t^7+102t^6+336t^5+656t^4+896t^3+848t^2+480t+128\\
&t^8+18t^7+122t^6+408t^5+792t^4+1080t^3+1024t^2+576t+128
\end{cases}
$$

$$
(12t^2+12)R_9(t)=
\begin{cases}
&t^7-6t^6+21t^5-60t^4+111t^3-162t^2+163t-108\\
&t^7-4t^6+19t^5-44t^4+95t^3-136t^2+149t-96\\
&t^7-2t^6+13t^5-28t^4+71t^3-110t^2+131t-84\\
&t^7+3t^5+12t^4-33t^3+84t^2-107t+72
\end{cases}
$$

$$
(3t^2-15)R_{10}(t)=
\begin{cases}
&t^7-6t^6+3t^5+30t^4-33t^3-27t^2+37t-27\\
&t^7-4t^6+t^5+16t^4-25t^3-10t^2+47t-12\\
&t^7-2t^6-5t^5+2t^4+23t^3+7t^2-43t+3\\
&t^7-15t^5+12t^4+39t^3-24t^2-17t-18
\end{cases}
$$

$$
(t^3-t^2-t+1)R_{11}(t)=
\begin{cases}
&t^4-6t^2-28t-39\\
&t^4+3t^3+11t^2+25t+32\\
&t^4+6t^3+20t^2+22t+23\\
&t^4+9t^3+21t^2+35t+6
\end{cases}
$$

$$
4t^2R_{12}(t)=
\begin{cases}
&t^3+3t^2+9t-9\\
&t^3+t^2-3t+9\\
&t^3-t^2-3t-9\\
&t^3-3t^2+9t+9
\end{cases}
$$

$$
4t^2R_{13}(t)=
\begin{cases}
&t^3-12t^2-72t+288\\
&t^3-4t^2-24t+288\\
&t^3+4t^2-24t-288\\
&t^3+12t^2-72t-288
\end{cases}
$$

$$
(240t^3-588240t^2+462160080t-113856219120)R_{14}(t)=
$$
$$
\begin{cases}
&t^4-1508t^3-616026t^2+1404632668t-262572118559\\
&t^4-2228t^3+1379094t^2-321457172t+190871636401\\
&t^4-2948t^3+2913414t^2-815367812t-172760883839\\
&t^4-3668t^3+3986934t^2-1072427252t-221781743279
\end{cases}
$$

$$
(240t^3+45360t^2-15574320t-3067284720)R_{15}(t)=
$$
$$
\begin{cases}
&t^4-1508t^3-616026t^2+381656668t+123089833441\\
&t^4-788t^3-249546t^2+327100108t+100683524881\\
&t^4-68t^3-343866t^2-264749252t-71708114879\\
&t^4+652t^3-898986t^2-398563412t-12874461839.
\end{cases}
$$

\section{Some basic properties of the sequence $(\xi(n,t))_n$}\label{ippo}

In this section we prove that we need only studying $\xi(n,t)$ for $t\ge0$ and we show that for fix $t\ge0$, the sequences $(\xi_i(n,t))_i$ and $(\xi_i(n,t))_n$ are strictly increasing sequences of positive integers.

A straightforward computation shows that for all $n\ge0$ and $t\in\Z$ we have 
\begin{equation}\label{sym}
\xi_4(n,t)=-\xi_1(n,-t-5)\qquad\textrm{and}\qquad
\xi_3(n,t)=-\xi_2(n,-t-5).
\end{equation}
In particular, since
$$
\xi(1,-2)=(0,1,-2,-3)\qquad\textrm{and}\qquad \xi(1,-1)=(-3,4,5,6) 
$$
are trivial sequences, also $\xi(1,-3)$ and $\xi(1,-4)$ are trivial sequences. 

\begin{lemma}\label{triv}
For each $n\ge0$, the B\"uchi sequences $\xi(n,-4)$, $\xi(n,-3)$, $\xi(n,-2)$ and $\xi(n,-1)$ are trivial sequences. 
\end{lemma}
\begin{proof}
From Theorem \ref{bigteo}, we have for $t=-1$
$$
\xi_i(n+2,-1)=2\xi_i(n+1,-1)-\xi_i(n,-1)
$$
for each $i=1,2,3,4$, with initial values for $n=0,1$ (reminding that $\xi(0,t)=(t+1,t+2,t+3,t+4)$):
$$
\begin{array}{cc}
\xi_1(0,-1)=0&\xi_1(1,-1)=-3\\
\xi_2(0,-1)=1&\xi_2(1,-1)=4\\
\xi_3(0,-1)=2&\xi_3(1,-1)=5\\
\xi_4(0,-1)=3&\xi_4(1,-1)=6;
\end{array}
$$
and for $t=-2$:
$$
\xi_i(n+2,-2)=-2\xi_i(n+1,-2)-\xi_i(n,-2)
$$
for each $i=1,2,3,4$, with initial values for $n=0,1$
$$
\begin{array}{cc}
\xi_1(0,-2)=-1&\xi_1(1,-2)=0\\
\xi_2(0,-2)=0&\xi_2(1,-2)=1\\
\xi_3(0,-2)=1&\xi_3(1,-2)=-2\\
\xi_4(0,-2)=2&\xi_4(1,-2)=-3.
\end{array}
$$
Solving the eight recurrence relations above, we obtain\,:
$$
\xi(n,-1)=(-3n,3n+1,3n+2,3n+3)\qquad\textrm{and}\qquad \xi(n,-2)=(-1)^{n}(n-1,-n,n+1,n+2)
$$
which are clearly trivial sequences. From Equations \eqref{sym}, we have
$$
\begin{aligned}
\xi_1(n,-3)=-\xi_4(n,-2)\qquad&\qquad\xi_2(n,-3)=-\xi_3(n,-2)\\
\xi_4(n,-3)=-\xi_1(n,-2)\qquad&\qquad\xi_3(n,-3)=-\xi_2(n,-2),
\end{aligned}
$$ 
hence 
$$
\xi(n,-3)=(-1)^n(-n-2,-n-1,n,-n+1), 
$$
and 
$$
\begin{aligned}
\xi_1(n,-4)=-\xi_4(n,-1)\qquad&\qquad\xi_2(n,-4)=-\xi_3(n,-1)\\
\xi_4(n,-4)=-\xi_1(n,-1)\qquad&\qquad\xi_3(n,-4)=-\xi_2(n,-1),
\end{aligned}
$$ 
hence 
$$
\xi(n,-4)=(-3n-3,-3n-2,-3n-1,3n)
$$
which are also clearly trivial sequences. 
\end{proof}

\begin{remark}
We deduce from Equations \eqref{sym} and Lemma \ref{triv} that it is enough to study the parametrizations $\xi(n,t)$ for $t\geq0$. 
\end{remark}

\begin{lemma}\label{crec}
If $(u_n)$ is a sequence of integers satisfying $u_{n+2}=\alpha u_{n+1}-u_n$ for each $n\geq0$, with $\alpha\ge 2$, and $u_1>u_0>0$, then for all $n\ge 1$ it satisfies $u_{n+1}>(\alpha-1)u_{n}>0$. 
\end{lemma}
\begin{proof}
We have 
$$
u_2=\alpha u_1-u_0=(\alpha-1)u_1+u_1-u_0>(\alpha-1)u_1>0.
$$
Suppose that $u_{n+1}>(\alpha-1)u_{n}>0$ for some $n\ge1$. We have
$$
u_{n+2}=\alpha u_{n+1}-u_{n}>\alpha u_{n+1}-\frac{u_{n+1}}{\alpha-1}\geq (\alpha-1)u_{n+1}.
$$
\end{proof}

\begin{corollary}
For each $t\ge0$ and for each $i=1,\dots,4$, we have 
$$
\xi_i(n+1,t)>(2t^2+10t+9)\xi_i(n,t).
$$
\end{corollary}
\begin{proof}
Fix $t\ge0$. We apply Lemma \ref{crec} to the sequence $u_n=\xi_i(n,t)$ for each $i=1,\dots,4$. By Item \ref{bigteorec} of Theorem \ref{bigteo}, $u_n$ satisfy the recurrence relation $u_{n+2}=\alpha u_{n+1}-u_n$, with 
$$
\alpha=f(t)=2t^2+10t+10\ge2.
$$ 
By Items \ref{bigteon0} and \ref{bigteon1} of Theorem \ref{bigteo}, we have
$$
u_1=
\begin{cases}
\xi_1(1,t)=2t^3+12t^2+19t+6>t+1=\xi_1(0,t)=u_0>0&\textrm{if } i=1\\
\xi_2(1,t)=2t^3+14t^2+31t+23>t+2=\xi_2(0,t)=u_0>0&\textrm{if } i=2\\
\xi_3(1,t)=2t^3+16t^2+41t+32>t+3=\xi_3(0,t)=u_0>0&\textrm{if } i=3\\
\xi_4(1,t)=2t^3+18t^2+49t+39>t+4=\xi_4(0,t)=u_0>0&\textrm{if } i=4
\end{cases}
$$
and we conclude in each case by Lemma \ref{crec}.
\end{proof}

\begin{lemma}\label{mixcrec}
If $(v_n)$ and $(w_n)$ are sequences of integers both satisfying the same recurrence relation $u_{n+2}=\alpha u_{n+1}-u_n$ for each $n\geq0$, with $\alpha\ge 2$, and $u_1>u_0>0$, and such that $w_0\ge v_0$ and $w_1-w_0>v_1-v_0$, then for all $n\ge 1$ we have $w_{n+1}-v_{n+1}>(\alpha-1)(w_{n}-v_n)>0$. 
\end{lemma}
\begin{proof}
We have 
$$
\begin{aligned}
w_2-v_2&=\alpha w_1-w_0-(\alpha v_1-v_0)\\
&=(\alpha-1)(w_1-v_1)+w_1-w_0-(v_1-v_0)\\
&>(\alpha-1)(w_1-v_1)\\
&>(\alpha-1)(w_0-v_0)\ge0.
\end{aligned}
$$
If for some $n\ge1$ we have $w_{n+1}-v_{n+1}>(\alpha-1)(w_{n}-v_n)>0$ then 
$$
\begin{aligned}
w_{n+2}-v_{n+2}&=\alpha w_{n+1}-w_n-(\alpha v_{n+1}-v_n)\\
&=\alpha(w_{n+1}-v_{n+1})-(w_n-v_n)\\
&>\alpha(w_{n+1}-v_{n+1})-\frac{w_{n+1}-v_{n+1}}{\alpha-1}\\
&>(\alpha-1)(w_{n+1}-v_{n+1})>0.
\end{aligned}
$$ 
\end{proof}

\begin{corollary}
For each $n$ and each $t\ge0$, the sequence $\xi(n,t)$ is a strictly increasing non-trivial B\"uchi sequence of positive integers. Moreover, for each $i=1,2,3$ and for each $n\ge 1$ we have 
$$
\xi_{i+1}(n+1,t)-\xi_i(n+1,t)>(2t^2+10t+9)(\xi_{i+1}(n,t)-\xi_i(n,t)).
$$
\end{corollary}
\begin{proof}
Fix $t\ge0$. We will apply Lemma \ref{mixcrec} to the sequences $v_n=\xi_1(n,t)$ and $w_n=\xi_2(n,t)$. By Item \ref{bigteorec} of Theorem \ref{bigteo}, both $v_n$ and $w_n$ satisfy the recurrence relation $u_{n+2}=\alpha u_{n+1}-u_n$, with 
$$
\alpha=f(t)=2t^2+10t+10\ge2.
$$ 
By Items \ref{bigteon0} and \ref{bigteon1} of Theorem \ref{bigteo}, we have
$$
v_1=\xi_1(1,t)=2t^3+12t^2+19t+6>t+1=\xi_1(0,t)=v_0>0,
$$
$$
w_1=\xi_2(1,t)=2t^3+14t^2+31t+23>t+2=\xi_2(0,t)=w_0>0,
$$
$$
w_0=\xi_2(0,t)=t+2> t+1=\xi_1(0,t)=v_0,
$$
and 
$$
w_1-w_0=2t^3+14t^2+30t+21>2t^3+12t^2+18t+5=v_1-v_0
$$
so all the hypothesis of Lemma \ref{mixcrec} are satisfied and we deduce that $\xi_2(n,t)-\xi_1(n,t)$ is a positive integer for each $n\ge0$ and that for each $n\ge 1$ we have
$$
\xi_{2}(n+1,t)-\xi_1(n+1,t)>(f(t)-1)(\xi_{2}(n,t)-\xi_1(n,t)).
$$
The two other cases are verified similarly. 
\end{proof}

\section{A family of curves associated to length $5$ solutions}\label{hyp}

Each length $4$ integer sequence $(x_1,x_2,x_3,x_4)$ might extend to the right or to the left. 
For given integers $n\ge1$ and $t\ge0$, a B\"uchi sequence $\xi(n,t)$ extends to the right if and only if $2\xi_4^2(n,t)-\xi_3^2(n,t)+2$ is a square, and it extends to the left if and only if $2\xi_1^2(n,t)-\xi_2^2(n,t)+2$ is a square. So for each $n\ge1$, we want to know whether or not the curves 
$$
y^2=2\xi_4^2(n,t)-\xi_3^2(n,t)+2\qquad(C_n^{\rm r})
$$
and 
$$
y^2=2\xi_1^2(n,t)-\xi_2^2(n,t)+2\qquad(C_n^{\rm \ell})
$$
have integer points at all. The polynomials on the right have degree $2(2n+1)=4n+2$. Unfortunately, we cannot prove that all these curves are hyperelliptic (we verified$\dagger$ it only up to $n=18$). If it were the case then we would already know that each of them has only finitely many integer points. If there are only finitely many B\"uchi sequences of length $5$, as we suspect, then all but finitely many of these curves has no integer point at all. 

Here is the equation for $(C_1^{\rm r})$
$$
y^2=4t^6+80t^5+620t^4+2400t^3+4905t^2+5020t+2020,
$$
the one for $(C_2^{\rm r})$
$$
16t^{10}+480t^9+6240t^8+46400t^7+218812t^6+684120t^5+1436320t^4+1999600t^3+1766797t^2+894990t+197505,
$$
and the one for $(C_3^{\rm r})$
$$
\begin{aligned}
&64t^{14}+2560t^{13}+46400t^{12}+505600t^{11}+3702416t^{10}+19280000t^9+73635280t^8+209537600t^7+\\
&446403560t^6+708503520t^5+824619920t^4+682516400t^3+379789209t^2+127204040t+19353040.
\end{aligned}
$$

\section{A List of non-parametrized integer points on $X_4$}\label{list}

In this section we list the strictly increasing sequences that are not given by any of the parametrizations presented in this paper. The first column is just the number of the line of the matrix. The graph is a plot of the two first columns. The quantity of points that we are not able to parametrize seems to get exponentially towards zero. 

$$
\left(\begin{array}{ccccc}
1 & 59 & 630 & 889 & 1088 \\
2 & 83 & 516 & 725 & 886 \\
3 & 108 & 6643 & 9394 & 11505 \\
4 & 108 & 707 & 994 & 1215 \\
5 & 177 & 878 & 1229 & 1500 \\
6 & 240 & 839 & 1162 & 1413 \\
7 & 287 & 11838 & 16739 & 20500 \\
8 & 311 & 752 & 1017 & 1226 \\
9 & 334 & 3693 & 5212 & 6379 \\
10 & 386 & 6237 & 8812 & 10789 \\
11 & 419 & 11020 & 15579 & 19078 \\
12 & 430 & 801 & 1048 & 1247 \\
13 & 477 & 3572 & 5029 & 6150 \\
14 & 510 & 1699 & 2348 & 2853 \\
15 & 514 & 1537 & 2112 & 2561 \\
16 & 570 & 7879 & 11128 & 13623 \\
17 & 601 & 4832 & 6807 & 8326 \\
18 & 862 & 1713 & 2264 & 2705 \\
19 & 883 & 25566 & 36145 & 44264 \\
20 & 916 & 26605 & 37614 & 46063 \\
21 & 1346 & 20353 & 28752 & 35201 \\
22 & 1546 & 5257 & 7272 & 8839 \\
23 & 1574 & 2693 & 3468 & 4099 \\
24 & 1616 & 3353 & 4458 & 5339 \\
25 & 1674 & 2695 & 3424 & 4023 \\
26 & 1766 & 8837 & 12372 & 15101 \\
27 & 1812 & 11587 & 16286 & 19905 \\
28 & 2066 & 6963 & 9628 & 11701 \\
29 & 2437 & 13062 & 18311 & 22360 \\
30 & 2477 & 15876 & 22315 & 27274 \\
31 & 2636 & 20685 & 29134 & 35633 \\
32 & 3048 & 5047 & 6454 & 7605 \\
33 & 3051 & 11578 & 16087 & 19584 \\
34 & 3247 & 9746 & 13395 & 16244 \\
35 & 3333 & 36682 & 51769 & 63360 \\
36 & 3673 & 5478 & 6821 & 7940 \\
37 & 4090 & 5701 & 6948 & 8003 \\
38 & 4743 & 36806 & 51835 & 63396 \\
39 & 5148 & 12253 & 16546 & 19935 \\
40 & 5331 & 15988 & 21973 & 26646 \\
41 & 5781 & 22342 & 31063 & 37824 \\
42 & 6449 & 25358 & 35277 & 42964 \\
43 & 6504 & 18065 & 24706 & 29907 \\
44 & 6756 & 33773 & 47282 & 57711 \\
45 & 7104 & 9823 & 11938 & 13731 \\
46 & 7234 & 24447 & 33808 & 41089 \\
47 & 7386 & 17033 & 22928 & 27591 
\end{array}\right)\qquad
\left(\begin{array}{ccccc}
48 & 7414 & 16875 & 22684 & 27283 \\
49 & 7594 & 10997 & 13572 & 15731 \\
50 & 7871 & 12162 & 15293 & 17884 \\
51 & 8562 & 17089 & 22600 & 27009 \\
52 & 9343 & 26408 & 36159 & 43790 \\
53 & 9741 & 19460 & 25739 & 30762 \\
54 & 9752 & 25249 & 34350 & 41501 \\
55 & 10888 & 25561 & 34470 & 41509 \\
56 & 11358 & 47107 & 65644 & 79995 \\
57 & 12129 & 18232 & 22753 & 26514 \\
58 & 12539 & 21430 & 27591 & 32608 \\
59 & 12710 & 46491 & 64508 & 78493 \\
60 & 13305 & 44986 & 62213 & 75612 \\
61 & 13500 & 29971 & 40178 & 48273 \\
62 & 13811 & 38380 & 52491 & 63542 \\
63 & 13835 & 33596 & 45453 & 54802 \\
64 & 13836 & 25693 & 33598 & 39969 \\
65 & 14416 & 40737 & 55778 & 67549 \\
66 & 14843 & 26758 & 34809 & 41320 \\
67 & 15369 & 52022 & 71947 & 87444 \\
68 & 15451 & 47988 & 66083 & 80194 \\
69 & 18793 & 33744 & 43865 & 52054 \\
70 & 20476 & 44445 & 59426 & 71327 \\
71 & 21648 & 38497 & 49954 & 59235 \\
72 & 21924 & 32243 & 39982 & 46449 \\
73 & 22377 & 45328 & 60071 & 71850 \\
74 & 23173 & 49926 & 66695 & 80024 \\
75 & 23174 & 56283 & 76148 & 91811 \\
76 & 25079 & 34122 & 41227 & 47276 \\
77 & 27283 & 57918 & 77231 & 92600 \\
78 & 27699 & 38828 & 47413 & 54666 \\
79 & 31659 & 51412 & 65453 & 76974 \\
80 & 33426 & 58483 & 75652 & 89589 \\
81 & 34030 & 59119 & 76368 & 90383 \\
82 & 45007 & 85256 & 111855 & 133246 \\
83 & 49040 & 61729 & 72222 & 81373 \\
84 & 50430 & 70781 & 86468 & 99717 \\
85 & 51077 & 89226 & 115385 & 136624 \\
86 & 53119 & 70562 & 84477 & 96404 \\
87 & 55506 & 72097 & 85528 & 97119 \\
88 & 58599 & 87328 & 108713 & 126534 \\
89 & 62429 & 86532 & 105253 & 121114 \\
90 & 63626 & 118165 & 154524 & 183827 \\
91 & 64776 & 98815 & 123826 & 144573 \\
92 & 68986 & 106617 & 134072 & 156791 \\
93 & 70143 & 94792 & 114241 & 130830 \\
94 & 77391 & 92440 & 105361 & 116862 
\end{array}\right)
$$

$$
\left(\begin{array}{ccccc}
95 & 78741 & 128278 & 163433 & 192264 \\
96 & 79292 & 91693 & 102606 & 112465 \\
97 & 80251 & 100090 & 116601 & 131048 \\
98 & 81770 & 131541 & 167092 & 196307 \\
99 & 98804 & 118755 & 135806 & 150943 \\
100 & 107366 & 169275 & 213964 & 250813 \\
101 & 108523 & 139124 & 164115 & 185774 \\
102 & 117178 & 144071 & 166680 & 186569 \\
103 & 138004 & 167365 & 192294 & 214343 \\
104 & 154097 & 200846 & 238605 & 271156 \\
105 & 154097 & 200846 & 238605 & 271156 \\
106 & 155730 & 226399 & 279752 & 324447 \\
107 & 158435 & 195324 & 226277 & 253478 \\
108 & 165267 & 222418 & 267631 & 306240 \\
109 & 183122 & 235379 & 277980 & 314869 \\
110 & 186101 & 246132 & 294157 & 335374 \\
111 & 225341 & 270018 & 308287 & 342304 \\
112 & 297422 & 352179 & 399500 & 441781 \\
113 & 311680 & 401551 & 474702 & 537997 \\
114 & 388048 & 447801 & 500470 & 548101 \\
115 & 421884 & 499235 & 566114 & 625887 \\
116 & 435682 & 484931 & 529620 & 570821 \\
117 & 646914 & 739327 & 821408 & 896001 \\
118 & 695001 & 761728 & 823063 & 880134 \\
119 & 740566 & 869223 & 981152 & 1081559 \\
120 & 839833 & 974682 & 1093019 & 1199740 \\
121 & 1052749 & 1157218 & 1253007 & 1341976
\end{array}\right)
$$

\includegraphics[width=400pt]{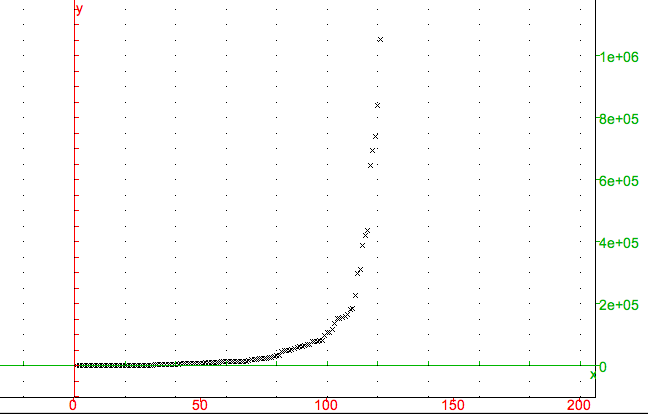}


\vspace{10pt}

\noindent Xavier Vidaux\\
Universidad de Concepci\'on\\
Facultad de Ciencias F\'isicas y Matem\'aticas\\
Departamento de Matem\'atica\\
Casilla 160 C\\
email: xvidaux@udec.cl

\end{document}